\documentclass{article}
\usepackage[T1]{fontenc}
\usepackage[utf8]{inputenc}
\usepackage{verbatim}

\usepackage{amsmath}
\usepackage{amsfonts}
\usepackage{amsthm}
\usepackage{amssymb}
\usepackage{bm}
\usepackage{enumerate}
\makeatletter
\numberwithin{equation}{section}
\numberwithin{figure}{section}
\theoremstyle{plain}
\newtheorem{thm}{\protect\theoremname}
  \theoremstyle{definition}
  \newtheorem{defn}[thm]{\protect\definitionname}
  \theoremstyle{plain}
  \newtheorem{fact}[thm]{\protect\factname}
  \theoremstyle{plain}
  \newtheorem{cor}[thm]{\protect\corollaryname}
  \theoremstyle{plain}
  \newtheorem{lem}[thm]{\protect\lemmaname}
   \theoremstyle{plain}
  \newtheorem{problem}[thm]{\protect\problemname}
   
\newenvironment{lyxlist}[1]
{\begin{list}{}
{\settowidth{\labelwidth}{#1}
 \setlength{\leftmargin}{\labelwidth}
 \addtolength{\leftmargin}{\labelsep}
 }}
{\end{list}}
  \theoremstyle{remark}
  \newtheorem{claim}[thm]{\protect\claimname}
  \theoremstyle{plain}
  \newtheorem{prop}[thm]{\protect\propositionname}
  \theoremstyle{remark}
  \newtheorem{rem}[thm]{\protect\remarkname}
  \theoremstyle{remark}
  \newtheorem*{claim*}{\protect\claimname}

\usepackage{a4wide}
\usepackage{hyperref}

\makeatother

  \providecommand{\claimname}{Claim}
  \providecommand{\corollaryname}{Corollary}
  \providecommand{\definitionname}{Definition}
  \providecommand{\factname}{Fact}
  \providecommand{\lemmaname}{Lemma}
  \providecommand{\propositionname}{Proposition}
  \providecommand{\remarkname}{Remark}
\providecommand{\theoremname}{Theorem}
\providecommand{\problemname}{Problem}

\def\Ind#1#2{#1\setbox0=\hbox{$#1x$}\kern\wd0\hbox to 0pt{\hss$#1\mid$\hss}
\lower.9\ht0\hbox to 0pt{\hss$#1\smile$\hss}\kern\wd0}
\def\Notind#1#2{#1\setbox0=\hbox{$#1x$}\kern\wd0\hbox to 0pt{\mathchardef
\nn="3236\hss$#1\nn$\kern1.4\wd0\hss}\hbox to 0pt{\hss$#1\mid$\hss}\lower.9\ht0
\hbox to 0pt{\hss$#1\smile$\hss}\kern\wd0}
\def\indi{\mathop{\mathpalette\Ind{}}}

\global\long\def\acl{\operatorname{acl}}

\global\long\def\inp{\operatorname{inp}}

\global\long\def\ind{\operatorname{\indi}}

\global\long\def\M{\operatorname{\mathbb{M}}}

\global\long\def\NTP{\operatorname{NTP}}

\global\long\def\NIP{\operatorname{NIP}}

\global\long\def\bdn{\operatorname{bdn}}

\global\long\def\dcl{\operatorname{dcl}}

\global\long\def\acl{\operatorname{acl}}

\global\long\def\rv{\operatorname{rv}}

\global\long\def\RV{\operatorname{RV}}

\global\long\def\ch{\operatorname{char}}

\global\long\def\ac{\operatorname{ac}}

\global\long\def\val{\operatorname{val}}

\global\long\def\WD{\operatorname{WD}}

\title{Henselian valued fields and $\inp$-minimality}

\author{Artem Chernikov   and Pierre Simon}

\begin{document}
\maketitle

\begin{abstract}

We prove that every ultraproduct of $p$-adics is $\inp$-minimal (i.e., of burden $1$). More generally, we prove an Ax-Kochen type result on preservation of $\inp$-minimality for Henselian valued fields of equicharacteristic $0$ in the $\RV$ language.
\end{abstract}

\section{Introduction}

In his work on the classification of first-order theories \cite{MR1083551}
Shelah has introduced a hierarchy of combinatorial properties of families
of definable sets, so called dividing lines, which includes stable
theories, simple theories, $\NIP$, NSOP, etc. An important line of
research in model theory is to characterize various algebraic structures
depending on their place in this classification hierarchy (this knowledge
can later be used to analyze various algebraic objects definable in
such structures using methods of generalized stability theory). Here
we will be concerned with valued fields and Ax-Kochen-type statements,
i.e. statements of the form ``a certain property of the valued field
can be determined by looking just at the value group and the residue
field''. For example, a classical theorem of Delon \cite{delon1978types}
shows that given a Henselian valued field of equicharacteristic $0$,
if the residue field is NIP, then the whole valued field is NIP. More
recent results of similar type are \cite{belair1999types} demonstrating
preservation of NIP for certain valued fields of positive characteristic,
\cite{MR3273451} demonstrating that the field of $p$-adics
is strongly dependent, and \cite{dolich2011dp} demonstrating that
it is in fact dp-minimal.

A motivating example for this article is to determine the model-theoretic
complexity of the theory of an ultraproduct of the fields of $p$-adics
$\mathbb{Q}_{p}$ modulo a non-principal ultrafilter on the set of
prime numbers. Namely, let $K=\prod\mathbb{Q}_{p}/\mathcal{U}$, where
$\mathcal{U}$ is a non-principal ultrafilter on the set of prime
numbers. Note that the residue field $k$ is a pseudo-finite field
of characteristic $0$ and that the value group $\Gamma$ is a $\mathbb{Z}$-group.
Besides, both $k$ and $\Gamma$ are interpretable in $K$ in the
pure ring language (e.g. by a result of Ax \cite{ax1965undecidability}).
This implies that the theory of $K$ is neither $\NIP$, nor simple
--- the two classes of structures extensively studied in model theory.
However it turns out that any ultraproduct of $p$-adics is $\NTP_{2}$
\cite{chernikov2014theories}. The class of $\NTP_{2}$ theories was
introduced by Shelah \cite[Chapter III]{MR1083551} and generalizes
both simple and $\NIP$ theories. We recall the definition.
\begin{defn}
Let $T$ be a complete first-order theory in a language $L$, and
let $\M\models T$ be a monster model. Let $\kappa$ be a cardinal
(finite or infinite).
\begin{enumerate}
\item An \emph{$\inp$-pattern} of depth $\kappa$ is given by $\left(\phi_{i}\left(x,y_{i}\right),\bar{a}_{i},k_{i}:i\in\kappa\right)$,
where $\phi_{i}\left(x,y_{i}\right)$ are $L$-formulas with a fixed
tuple of free variables $x$ and a varying tuple of parameter variables
$y_{i}$, $\bar{a}_{i}=\left(a_{i,j}:j\in\omega\right)$ are sequences
of tuples of elements from $\M$, and $k_{i}$ are natural numbers
such that:

\begin{enumerate}
\item For every $i\in\kappa$, the set $\left\{ \phi_{i}\left(x,a_{i,j}\right)\right\} _{j\in\omega}$
is $k_{i}$-inconsistent (i.e. no subset of size $\geq k_{i}$ is
consistent).
\item For every $f:\kappa\to\omega$, the set $\left\{ \phi_{i}\left(x,a_{i,f\left(i\right)}\right)\right\} _{i\in\kappa}$
is consistent.
\end{enumerate}
\item $T$ is $\NTP_{2}$ if there is a (cardinal) bound on the depths of
$\inp$-patterns.
\end{enumerate}
\end{defn}
Other algebraic examples of $\NTP_{2}$ structures were identified
recently, including bounded pseudo real closed and pseudo $p$-adically
closed fields \cite{montenegro2014pseudo}, certain model complete
multi-valued fields \cite{johnson2016fun} and certain valued
difference fields, e.g. the theory $\mbox{VFA}_{0}$ of a non-standard
Frobenius on an algebraically closed valued field of characteristic
zero \cite{chernikov2014valued}. See also \cite{chernikov2015groups} and \cite{hempel2015groups} 
for some general results about groups and fields definable in $\NTP_{2}$
structures.\\

The notion of \emph{burden} was introduced by Adler \cite{adler2007strong}
based on Shelah's cardinal invariant $\kappa_{\inp}$ and provides
a quantitative refinement of $\NTP_{2}$. In the special case of simple
theories burden corresponds to preweight, and in the case of $\NIP$
theories to dp-rank (e.g. see \cite[Section 3]{chernikov2014theories}
for the details and references).
\begin{defn}

\begin{enumerate}
\item $T$ is \emph{strong} if there are no $\inp$-patterns of infinite
depth.
\item $T$ is \emph{of finite burden} if there are no $\inp$-patterns of
arbitrary large finite depth, with $x$ a singleton.
\item $T$ is \emph{$\inp$-minimal} if there is no $\inp$-pattern of depth
$2$, with $x$ a singleton.
\end{enumerate}
\end{defn}
Note that $\inp$-minimality implies finite burden implies strong (the last implication uses submultiplicativity of burden from \cite{chernikov2014theories}). All the examples mentioned above have been demonstrated to be strong
of finite burden, with the exception of $\mbox{VFA}_{0}$:
it remains open if $\mbox{VFA}_{0}$ is strong, see \cite[Question 5.2]{chernikov2014valued}. Some results about strong groups and fields can be found in \cite[Section 4]{chernikov2015groups} and \cite{dolich2015strong}.

Returning to ultraproducts of $p$-adics, we have the following
more general result.
\begin{fact}
\label{fac: AxKochenNTP2}\cite{chernikov2014theories} Let $\bar{K}=\left(K,k,\Gamma,\val,\ac\right)$
be a Henselian valued field of equicharacteristic $0$, considered
as a three-sorted structure in the Denef-Pas language $L_{\ac}$ (i.e.
there is a sort $K$ for the field itself, as well as sorts $k$ for
the residue field and $\Gamma$ for the value group, together with
the maps $v:K\to\Gamma$ for the valuation and $\ac:K\to k$ for an
angular component).
\begin{enumerate}
\item If $k$ is $\NTP_{2}$, then $\bar{K}$ is $\NTP_{2}$.
\item If both $k$ and $\Gamma$ are strong (of finite burden) then $\bar{K}$
is strong (respectively, of finite burden).
\end{enumerate}
\end{fact}
Any pseudofinite field is supersimple of SU-rank $1$, so in particular
is $\inp$-minimal. Any ordered $\mathbb{Z}$-group is dp-minimal,
so in particular is $\inp$-minimal. It follows that any ultraproduct
of $p$-adics is strong, of finite burden. However, Fact \ref{fac: AxKochenNTP2}(2)
gives a finite bound on the burden of $\bar{K}$ in terms of the burdens
of $k$ and $\Gamma$ via a certain Ramsey number, and is far from
optimal in general. It was conjectured in \cite[Problem 7.13]{chernikov2014theories}
that all ultraproducts of $p$-adics in the pure ring language are
$\inp$-minimal (note that in the Denef-Pas language, no valued field
with an infinite residue field can be $\inp$-minimal as $\left\{ \ac\left(x\right)=a_{i}\right\} ,\left\{ \val\left(x\right)=v_{i}\right\} $
with $\left(a_{i}\right),\left(v_{i}\right)$ pairwise different give
an $\inp$-pattern of depth $2$).

In this paper we establish an Ax-Kochen type result for
$\inp$-minimality in the $\RV$ language for valued fields, in particular
confirming that conjecture.
\begin{thm}
\label{thm: main}Let $\bar{K}=\left(K,\RV,\rv\right)$ be a Henselian
valued field of equicharacteristic $0$, viewed as a structure in
the $\RV$-language (see Section \ref{sec: Reduction to RV}). 
Assume that both the residue field $k$
and the value group $\Gamma$ are $\inp$-minimal, and that moreover $k^\times / (k^\times)^p$ is finite for all prime $p$. Then $\bar{K}$ is $\inp$-minimal.\end{thm}
\begin{cor}
Any ultraproduct of $p$-adics is $\inp$-minimal.
\end{cor}

Recall the following definition, see e.g. \cite{simon2011dp}.

\begin{defn} \label{def: dp-min}
A theory is \emph{dp-minimal} if for every mutually indiscernible sequences of tuples $(a_i : i \in \omega), (a'_i : i \in \omega)$ and a singleton $b$ in the home sort, one of this sequences must be indiscernible over $b$.
\end{defn}
\begin{rem} \label{rem: dp-min}
An $\NIP$ theory is \emph{dp-minimal} if and only if it is $\inp$-minimal.
\end{rem}

Johnson \cite{johnson2018canonical} shows that a dp-minimal not strongly minimal field admits a definable Henselian valuation. It follows that if $K$ is dp-minimal, then $K^\times / (K^\times)^p$ is finite for all prime $p$ (a fact which Johnson states and uses).
Combining this with Delon's result on preservation
of NIP we have the following corollary (which also appears in Johnson's thesis \cite{johnson2016fun}).
\begin{cor}
\label{cor: dp-min case} Under the same assumptions on $\bar{K}$, if both $k$ and $\Gamma $ are dp-minimal, then $\bar{K}$ is dp-minimal.
\end{cor}

There are three steps in the proof of the main theorem, corresponding
to the sections of the paper. First, we recall some facts about the
$\RV$ setting and show that the whole valued field is $\inp$-minimal
if and only if the $\RV$ sort is $\inp$-minimal. Second, we show
that the $\RV$ sort eliminates quantifiers down to the residue field
$k$ and the value group $\Gamma$. Using this quantifier
elimination, in the last section we show that the $\RV$ sort is $\inp$-minimal if and
only if both $k$ and $\Gamma$ are $\inp$-minimal. Finally, we discuss some problems and future research directions.

\section{\label{sec: Reduction to RV}Reduction to $\RV$}

We recall some basic facts about the RV setting, we are going to use \cite{flenner2011relative} as a reference. Fix a valued field $K$, with value
group $\Gamma$ and residue field $k$. Let $\RV$ be the quotient
group $K^{\times}/\left(1+\mathfrak{m}\right)$ where $\mathfrak{m}=\left\{ x\in K:\val\left(x\right)>0\right\} $
is the maximal ideal of the valuation ring. We have a short exact
sequence $1\to k^{\times}\to\RV\overset{\val_{\rv}}{\to}\Gamma\to0\mbox{.}$

Consider now the two-sorted structure $\bar{K}=\left(K,\RV,\rv\right)$
in the language $L_{\RV^{+}}$ consisting of: 
\begin{itemize}
\item the quotient map $\rv:K\to\RV$,
\item on the sort $K$, the ring structure,
\item on the sort $\RV$, the structure $\cdot,1$ of a multiplicative group,
a symbol $0$, a symbol $\infty$ and a ternary relation $\oplus$.
\\
The multiplicative group structure is interpreted as the group structure
induced from $K^{\times}$ and $0\cdot x=x\cdot0=0$, $\infty = \rv(0)$. The relation $\oplus$ is interpreted as the partially defined addition inherited from $K$: $\oplus(a,b,c) \iff \exists x,y,z \in K \left( a =\rv(x) \land b = \rv(y) \land c = \rv(z) \land x+y = z \right)$.

\end{itemize}

\begin{rem}\label{rem: WD}
\begin{enumerate}

\item One can define the set $\WD(x,y)$ of pairs of elements for which the sum is well-defined as $\forall z, z' (\oplus(x,y,z) \land \oplus(x,y,z') \implies z=z')$. Given a pair of elements $x,y \in \RV$ such that $\WD(x,y)$ holds, we write $x+y$ to denote the unique element $z \in \RV$ satisfying $\oplus(x,y,z)$. 
\item We have $\WD(\rv(a),\rv(b)) \iff \val(a+b) = \min \left\{\val(a), \val(b) \right\}$, in which case $\rv(a+b) = \rv(a) + \rv(b)$ (see \cite[Proposition 2.4]{flenner2011relative}).

\item The relation $\val_{\rv}(x) \leq \val_{\rv}(y)$  on $\RV$ is definable in this language \cite[Proposition 2.8(1)]{flenner2011relative}. Namely, let $d \in \RV$ be arbitrary with $\val_{\rv}(d) = 0$. Then  $\val_{\rv}(x) > 0 \iff dx + 1 = 1$, and $\val_{\rv}(x) = 0 \iff \neg \val_{\rv}(x) > 0 \land \exists y (x \cdot y = 1 \land \neg \val_{\rv}(y) > 0)$. Then $\val_{\rv}(x) = \val_{\rv}(y) \iff \exists u (\val_{\rv}(u) = 0 \land x = u \cdot y)$ and $\val_{\rv}(x) < \val_{\rv}(y) \iff x \neq \infty \land x + dy = x$.
\end{enumerate}
\end{rem}

Let $\bar{\mathbb{K}}\succ\bar{K}$ be a monster model. We may always
assume that $\bar{\mathbb{K}}$ admits a cross-section map $\ac:K\to k^{\times}$,
so we can view $\bar{\mathbb{K}}$ also as a structure in the language
$L_{\ac}$ with $\ac$ added to the language.
\begin{fact}
\label{fac: Flenner's cell decomposition}\cite[Proposition 5.1]{flenner2011relative}
\begin{enumerate}
\item Let $K$ be a Henselian valued field with $\ch\left(k\right)=0$,
and suppose that $S\subseteq K$ is definable. Then there are $\alpha_{1},\ldots,\alpha_{k}$
and a definable subset $D\subseteq\RV^{k}$ such that 
\[
S=\left\{ x\in K:\left(\rv\left(x-\alpha_{1}\right),\ldots,\rv\left(x-\alpha_{k}\right)\right)\in D\right\} \mbox{.}
\]

\item The $\RV$ sort is fully stably embedded (i.e. the structure on $\RV$
induced from $\bar{K}$, with parameters, is precisely the one described above).
\end{enumerate}
\end{fact}
The following two lemmas are easy to verify (see \cite{chernikov2010indiscernible}, or the proof of \cite[Claim 1.17]{MR3273451}
for the details).
\begin{lem}
\label{lem: pseudo-conv or fan}Let $\left(a_{i}\right)_{i\in I}$
be an $L_{\ac}$-indiscernible sequence of singletons in $\mathbb{K}$,
and consider the function $\left(i,j\right)\mapsto\val\left(a_{j}-a_{i}\right)$
for $i<j\in I$. Then one of the following cases occurs:
\begin{enumerate}
\item It is strictly increasing depending only on $i$ (so the sequence
is pseudo-convergent).
\item It is strictly decreasing depending only on $j$ (so the sequence
taken in the reverse direction is pseudo-convergent).
\item It is constant (we'll refer to such a sequence as a ``fan'').
\end{enumerate}
\end{lem}

\begin{lem}
\label{lem: cutting a pseudo-convergent sequence}Let $(a_{i})_{i\in I}$
be an $L_{\ac}$-indiscernible pseudo-convergent sequence from $\mathbb{K}$.
Then for any $d\in\mathbb{K}$ there is some $i_{*}\in\bar{I}\cup\{+\infty,-\infty\}$
(where $\bar{I}$ is the Dedekind closure of $I$) such that (taking
$a_{\infty}$ from $\mathbb{K}$ such that $I\frown a_{\infty}$ is
indiscernible):
\begin{lyxlist}{00.00.0000}
\item [{For~$i<i_{*}$:}] $\val(a_{\infty}-a_{i})<\val(d-a_{\infty})$,
$\val(d-a_{i})=\val(a_{\infty}-a_{i})$ and $\ac(d-a_{i})=\ac(a_{\infty}-a_{i})$.
\item [{For~$i>i_{*}$:}] $\val(a_{\infty}-a_{i})>\val(d-a_{\infty})$,
$\val(d-a_{i})=\val(d-a_{\infty})$ and $\ac(d-a_{i})=\ac(d-a_{\infty})$.
\end{lyxlist}
\end{lem}
\begin{rem}\label{rem: basic rv}
Note also that for any non-zero $x,y\in K$, $\rv\left(x\right)=\rv\left(y\right)$
if and only if $\val\left(x-y\right)>\val\left(y\right)$; and for
any $z\in K$ and $x,y\in K\setminus\left\{ z\right\} $, $\rv\left(x-z\right)=\rv\left(y-z\right)$
if and only if $\val\left(x-y\right)>\val\left(y-z\right)$.
\end{rem}

~

In the remainder of this section we will reduce $\inp$-minimality
of $\bar{K}$ to $\inp$-minimality of the $\RV$ sort with the induced
structure.

First we treat a key special case. Assume that there is an $\inp$-pattern
consisting of formulas $\psi\left(x,yz\right)=\phi\left(\rv\left(x-y\right),z\right)$
and $\psi'\left(x,yz'\right)=\phi'\left(\rv\left(x-y\right),z'\right)$
and \emph{mutually $L_{\ac}$-indiscernible} sequences $\left(c_{i}\right)_{i\in\mathbb{Z}},\left(c_{i}'\right)_{i\in\mathbb{Z}}$
with $c_{i}=a_{i}\widehat{\,}b_{i}$ and $c_{i}'=a_{i}'\widehat{\,}b_{i}'$
where $\phi$ and $\phi'$ are $\RV$-formulas, $b_{i}\in\RV^{\left|z\right|},b_{i}'\in\RV^{\left|z'\right|}$
and $a_{i},a_{i}'\in K$. Without loss of generality both $\left\{ \phi\left(\rv\left(x-a_{i}\right),b_{i}\right)\right\} _{i\in\mathbb{Z}}$
and $\left\{ \phi'\left(\rv\left(x-a_{i}'\right),b_{i}'\right)\right\} _{i\in\mathbb{Z}}$
are $k$-inconsistent, and let $d\models\phi\left(\rv\left(x-a_{0}\right),b_{0}\right)\land\phi'\left(\rv\left(x-a_{0}'\right),b_{0}'\right)$.
We may also add to the base elements $a_{\infty},a_{-\infty},a_{\infty}',a_{-\infty}'$
continuing our sequences on the left and on the right.
\begin{claim}
$\val\left(d-a_{i}\right)\leq\val\left(d-a_{0}'\right)$ and $\val\left(d-a_{j}'\right)\leq\val\left(d-a_{0}\right)$
for all $i$ and $j$. In particular, $\val\left(d-a_{0}\right)=\val\left(d-a_{0}'\right)=\gamma$ for some $\gamma \in \Gamma$.\end{claim}
\begin{proof}
Assume that $\val\left(d-a_{i}\right)>\val\left(d-a_{0}'\right)$
for some $i$. Then $\rv\left(d-a_{0}'\right)=\rv\left(a_{i}-a_{0}'\right)$.
So $\models\phi'\left(\rv\left(a_{i}-a_{0}'\right),b_{0}'\right)$,
and by mutual indiscernibility $a_{i}\models\left\{ \phi'\left(\rv\left(x-a_{j}'\right),b_{j}'\right)\right\} _{j\in\omega}$
--- a contradiction. The other part is by symmetry.\end{proof}
\begin{claim}
\label{cla: gamma is below}$\gamma\leq\val\left(a_{0}-a_{0}'\right)$.\end{claim}
\begin{proof}
As otherwise $\val\left(d-a_{0}\right)=\val\left(d-a_{0}'\right)=\gamma>\val\left(a_{0}-a_{0}'\right)$, hence $\val(a_0 - a'_0) = \val((d-a'_0) - (a_0 - a'_0))$ $=\val\left(d-a_{0}\right)$
--- a contradiction.
\end{proof}

\noindent
We now consider several cases separately.

\smallskip
\noindent
\textbf{Case A}: $\val\left(a_{i}-a_{j}'\right)$ is constant, equal
to some $\gamma'\in\Gamma$.

\smallskip
As in this case the two sequences are mutually indiscernible over
$\gamma'$, we may add it to the base.
%
Note that $\gamma\leq\gamma'$ by Claim \ref{cla: gamma is below}.
The following subcases cover all the possible situations, using mutual
indiscernibility of the sequences over $\gamma'$.

\smallskip \noindent
\textbf{Subcase 1}: $\gamma<\gamma'$.

\smallskip
Then $\rv\left(d-a_{i}\right)=\rv\left(d-a_{j}'\right)=\alpha$ for
all $i,j$, for some some $\alpha\in\RV$ with $\val_{\rv}\left(\alpha\right)=\gamma$.
Note furthermore that for any $\alpha^{*}\in\RV$ such that $\val_{\rv}\left(\alpha^{*}\right)<\gamma'$
we can find some $d^{*}\in K$ such that $\rv\left(d^{*}-a_{i}\right)=\rv\left(d^{*}-a_{i}'\right)=\alpha^{*}$.

But then consider the array $$\widetilde{\phi}\left(\widetilde{x},b_{i}\right)=\phi\left(\widetilde{x},b_{i}\right)\land\val_{\rv}\left(\widetilde{x}\right)<\gamma',$$
$$\widetilde{\phi}'\left(\widetilde{x},b_{i}'\right)=\phi'\left(\widetilde{x},b'_{i}\right)\land\val_{\rv}\left(\widetilde{x}\right)<\gamma',$$
where $\widetilde{x}$ and $b_{i},b_{i}'$ are ranging over the $\RV$
sort and $\widetilde{\phi},\widetilde{\phi}'$ are $\RV$-formulas
(we are abusing the notation by writing $\val_{\rv}\left(\widetilde{x}\right)<\gamma'$
as a shortcut for $\val_{\rv}(\tilde{x}) < \val_{\rv}\left(a_{\infty}-a_{\infty}'\right)$).
We have $\models\widetilde{\phi}\left(\alpha,b_{0}\right)\land\widetilde{\phi}'\left(\alpha,b_{0}'\right)$
and $\left\{ \widetilde{\phi}\left(\widetilde{x},b_{i}\right)\right\} _{i\in\mathbb{Z}},\left\{ \widetilde{\phi}'\left(\widetilde{x},b_{i}'\right)\right\} _{i\in\mathbb{Z}}$
are both inconsistent by the previous observation as the original
array was inconsistent. This gives us an $\inp$-pattern in the structure induced on the $\RV$ sort, and so implies that $\RV$
is not $\inp$-minimal.

\smallskip \noindent
\textbf{Subcase 2}: $\gamma=\gamma'$, $\val\left(a_{i}-a_{j}\right)>\gamma$
and $\val\left(a_{i}'-a_{j}'\right)>\gamma$ for all $i<j$.

\smallskip
It follows by Remark \ref{rem: basic rv} that there are $\alpha,\alpha'\in\RV$ with $\val_{\rv}\left(\alpha\right)=\val_{\rv}\left(\alpha'\right)=\gamma$
such that $\rv\left(d-a_{i}\right)=\alpha$ and $\rv\left(d-a_{i}'\right)=\alpha'$
for all $i$. Furthermore, $\rv\left(a_{i}-a_{j}'\right)=\alpha'-\alpha=:\beta$
for all $i,j$. It follows that our sequences are mutually indiscernible
over $\beta$ and we can add it to the base. 

We then consider a new array $$\widetilde{\phi}\left(\widetilde{x},b_{i}\right)=\phi\left(\widetilde{x},b_{i}\right)\land\val_{\rv}\left(\widetilde{x}\right)=\gamma,$$
$$\widetilde{\phi}'\left(\widetilde{x},b_{i}'\right)=\phi'\left(\widetilde{x}-\beta,b_{i}'\right)\land\val_{\rv}\left(\widetilde{x}\right)=\gamma.$$
It follows that $\alpha\models\widetilde{\phi}\left(\widetilde{x},b_{0}\right)\land\widetilde{\phi}'\left(\widetilde{x},b_{0}'\right)$, so to contradict $\inp$-minimality of $\RV$ it is enough to show that $\left\{ \widetilde{\phi}\left(\widetilde{x},b_{i}\right)\right\} _{i\in\mathbb{Z}},\left\{ \widetilde{\phi}'\left(\widetilde{x},b_{i}'\right)\right\} _{i\in\mathbb{Z}}$
are both inconsistent. Let $\alpha^{*}\in\RV$
with $\val_{\rv}\left(\alpha^{*}\right)=\gamma$ be arbitrary, and take $d^* \in K$ such that $\rv(d^* - a_0) = \alpha^*$. Using Remark \ref{rem: basic rv} again, we then have $\rv\left(d^{*}-a_{i}\right)=\alpha^{*}$ and $\rv\left(d^{*}-a_{i}'\right)=\alpha^{*}+\beta$ for all $i$. Hence any $\alpha^*$ realizing a row in the new array gives $d^*$ realizing a row in the original array.

\smallskip \noindent
\textbf{Subcase 3}: $\gamma=\gamma'$, $\val\left(a_{i}-a_{j}\right)>\gamma$
and $\val\left(a_{i}'-a_{j}'\right)=\gamma$ for all $i<j$.

\smallskip
In this case we still have some $\alpha\in\RV$ such that $\rv\left(d-a_{i}\right)=\alpha$
for all $i$. On the other hand, it follows that $\rv\left(d-a_{i}'\right)=\rv\left(d-a_{\infty}\right)+\rv\left(a_{\infty}-a_{i}'\right)$.

We then consider a new array given by $$\widetilde{\phi}\left(\widetilde{x},b_{i}\right)=\phi\left(\widetilde{x},b_{i}\right)\land\val_{\rv}\left(\widetilde{x}\right)=\gamma,$$
$$\widetilde{\phi}'\left(\widetilde{x},\widetilde{b}_{i}'\right)=\phi'\left(\widetilde{x}+\rv\left(a_{\infty}-a_{i}'\right),b_{i}'\right)\land\val_{\rv}\left(\widetilde{x}\right)=\gamma\land \WD \left( \tilde{x}, \rv\left(a_{\infty}-a_{i}'\right)\right),$$
so $\widetilde{b}_{i}'=\rv\left(a_{\infty}-a_{i}'\right)\widehat{\,}b_{i}'$.
Note that $\left(b_{i}\right)_{i\in\mathbb{Z}}$ and $\left(\widetilde{b}_{i}'\right)_{i\in\mathbb{Z}}$
are mutually indiscernible sequences in $\RV$. It follows that $\alpha\models\widetilde{\phi}\left(\widetilde{x},b_{0}\right)\land\widetilde{\phi}'\left(\widetilde{x},\widetilde{b}_{0}'\right)$, hence to contradict $\inp$-minimality of $\RV$ it is enough to show that both $\left\{ \widetilde{\phi}\left(\widetilde{x},b_{i}\right)\right\} _{i\in\mathbb{Z}},\left\{ \widetilde{\phi}'\left(\widetilde{x},\widetilde{b}_{i}'\right)\right\} _{i\in\mathbb{Z}}$
are inconsistent.
Let $\alpha^{*}\in\RV$ be arbitrary such that $\val_{\rv}\left(\alpha^{*}\right)=\gamma$ and $\WD(\alpha^*, \rv(a_\infty - a'_i))$ for all $i$.
Let $d^{*} \in K$ be such that $\rv(d^* - a_\infty) = \alpha^*$. Then $\rv\left(d^{*}-a_{i}\right)=\alpha^{*}$
and $\rv\left(d^{*}-a_{i}'\right)=\alpha^{*}+\rv\left(a_{\infty}-a_{i}'\right)$ for all $i$. This implies that for any $\alpha^*$ realizing a row in the new array, the corresponding $d^*$ realizes the same row in the original array.

\smallskip \noindent
\textbf{Subcase 4}: $\gamma=\gamma'$, $\val\left(a_{i}-a_{j}\right)=\val\left(a_{i}'-a_{j}'\right)=\gamma$
for all $i<j$.

\smallskip
Then $\rv\left(d-a_{i}\right)=\rv\left(d-a_{\infty}\right)+\rv\left(a_{\infty}-a_{i}\right)$
and $\rv\left(d-a_{i}'\right)=\rv\left(d-a_{\infty}\right)+\rv\left(a_{\infty}-a_{i}'\right)$
(as $\val\left(d-a_{i}'\right)=\val\left(d-a_{i}\right)=\val\left(d-a_{\infty}\right)=\val\left(a_{\infty}-a_{i}'\right)$, because the first three are equal to $\gamma$ and the last one to $\gamma'$).

We consider a new array given by
$$\widetilde{\phi}\left(\widetilde{x},\widetilde{b}_{i}\right)=\phi\left(\widetilde{x}+\rv\left(a_{\infty}-a_{i}\right),b_{i}\right)\land\val_{\rv}\left(\widetilde{x}\right)=\gamma\land \WD \left( \widetilde{x}, \rv\left(a_{\infty}-a_{i}\right)\right),$$
$$\widetilde{\phi}'\left(\widetilde{x},\widetilde{b}_{i}'\right)=\phi'\left(\widetilde{x}+\rv\left(a_{\infty}-a_{i}'\right),b_{i}'\right)\land\val_{\rv}\left(\widetilde{x}\right)=\gamma\land \WD \left(\widetilde{x}, \rv\left(a_{\infty}-a_{i}'\right)\right),$$
so $\widetilde{b}_{i}=\rv\left(a_{\infty}-a_{i}\right)\widehat{\,}b_{i}$
and $\widetilde{b}_{i}'=\rv\left(a_{\infty}-a_{i}'\right)\widehat{\,}b_{i}'$.
Note that $\left(\widetilde{b}_{i}\right)_{i\in\mathbb{Z}}$ and $\left(\widetilde{b}_{i}'\right)_{i\in\mathbb{Z}}$
are mutually indiscernible sequences in $\RV$. It follows that $\alpha\models\widetilde{\phi}\left(\widetilde{x},\widetilde{b}_{0}\right)\land\widetilde{\phi}'\left(\widetilde{x},\widetilde{b}_{0}'\right)$, so to contradict $\inp$-minimality of $\RV$ it is enough to show that both $\left\{ \widetilde{\phi}\left(\widetilde{x},\widetilde{b}_{i}\right)\right\} _{i\in\mathbb{Z}},\left\{ \widetilde{\phi}'\left(\widetilde{x},\widetilde{b}_{i}'\right)\right\} _{i\in\mathbb{Z}}$
are inconsistent. Let $\alpha^{*}\in\RV$ be arbitrary such that $\val_{\rv}\left(\alpha^{*}\right)=\gamma$.
Let $d^{*}$ be such that $\rv(d^* - a_\infty) = \alpha^*$. Then $\rv\left(d^{*}-a_{i}\right)=\alpha^{*}+\rv\left(a_{\infty}-a_{i}\right)$ and 
$\rv\left(d^{*}-a_{i}'\right)=\alpha^{*}+\rv\left(a_{\infty}-a_{i}'\right)$ for all $i$, assuming these sums are well-defined (see Remark \ref{rem: WD}). But this implies that for any $\alpha^*$ realizing a row in the new array (hence all the sums above corresponding to this row are well-defined by the choice of $\widetilde{\phi}, \widetilde{\phi}'$), the corresponding $d^*$ realizes the same row in the original array.

\smallskip \noindent
\textbf{Subcase 5}: $\gamma=\gamma'$, $\val\left(a_{i}-a_{i}\right)=\gamma$
and $\val\left(a_{i}'-a_{j}'\right)>\gamma$ for all $i<j$.

\smallskip
Follows from Subcase 3 by symmetry.

\medskip \noindent
\textbf{Case B}: Not Case A.
\begin{claim}
At least one of the sequences $\left(a_{i}\right)_{i\in\mathbb{Z}}$,
$\left(a_{i}'\right)_{i\in\mathbb{Z}}$ is not a fan.\end{claim}
\begin{proof}
Assume that both are, say $\val\left(a_{i}-a_{j}\right)=\alpha$ and
$\val\left(a_{i}'-a_{j}'\right)=\alpha'$ for all $i<j$. It follows
by mutual indiscernibility that $\val\left(a_{i}-a_{j}'\right)\leq\min\left\{ \alpha,\alpha'\right\} $
for all $i,j$. But then $\val\left(a_{i}-a_{j}'\right)=\val\left(a_{0}-a_{0}'\right)$
for all $i,j$, thus putting us in Case A.
\end{proof}
So we may assume that $\left(a_{i}\right)_{i\in\mathbb{Z}}$ is a
pseudo-convergent sequence (by Lemma \ref{lem: pseudo-conv or fan},
possibly exchanging $\left(a_{i}\right)$ with $\left(a_{i}'\right)$
and reverting the ordering of the sequence).

\smallskip \noindent
\textbf{Subcase 1}: Some (equivalently, every) $a_{i}'$ is a pseudo-limit
of $\left(a_{i}\right)_{i\in\mathbb{Z}}$.

\smallskip
Then $\rv\left(d-a_{i}'\right)=\rv\left(d-a_{\infty}\right)$ for
all $i$ (by Claim \ref{cla: gamma is below}).

We define $\widetilde{\phi}'\left(\widetilde{x},b_{i}'\right)=\phi'\left(\widetilde{x},b_{i}'\right)\land\val_{\rv}\left(\widetilde{x}\right)<\val\left(a_{\infty}-a_{\infty}'\right)$.

By Lemma \ref{lem: cutting a pseudo-convergent sequence} it follows
that there is some $i^{*}\in\left\{ -\infty\right\} \cup\mathbb{Z}\cup\left\{ \infty\right\} $
such that $\rv\left(d-a_{i}\right)=\rv\left(d-a_{\infty}\right)$
for $i>i^{*}$ and $\rv\left(d-a_{i}\right)=\rv\left(a_{\infty}-a_{i}\right)$
for $i<i^{*}$. Again by Claim \ref{cla: gamma is below}, $i^{*}\leq0$.
Let's restrict $\left(a_{i}\right)_{i\in\mathbb{Z}}$ to $\left(a_{i}\right)_{i\in\omega}$.

If $\val\left(d-a_{\infty}\right)<\val\left(a_{\infty}-a_{0}\right)$
then $\rv\left(d-a_{i}\right)=\rv\left(d-a_{\infty}\right)$ for all
$i$. If $\val\left(d-a_{\infty}\right)=\val\left(a_{\infty}-a_{0}\right)$
then $\rv\left(d-a_{i}\right)=\rv\left(d-a_{\infty}\right)$ for all
$i>0$ and $\rv\left(d-a_{0}\right)=\rv\left(d-a_{\infty}\right)+\rv\left(a_{\infty}-a_{0}\right)$.
We thus define 
\begin{eqnarray*}
\widetilde{\phi}\left(\widetilde{x},\widetilde{b}_{i}\right) & = & \left(\val\left(a_{i}-a_{\infty}\right)>\val_{\rv}\left(\widetilde{x}\right)\land\phi\left(\widetilde{x},b_{i}\right)\right)\lor\\
 &  & \lor\left(\val\left(a_{i}-a_{\infty}\right)=\val_{\rv}\left(\widetilde{x}\right)\land\WD(\widetilde{x},\rv\left(a_{\infty}-a_{i}\right))\land\phi\left(\widetilde{x}+\rv\left(a_{\infty}-a_{i}\right),b_{i}\right)\right)
\end{eqnarray*}

with $\widetilde{b}_{i}=b_{i}\hat{\,}\rv\left(a_{i}-a_{\infty}\right)$.
Then $\left(\widetilde{b}_{i}\right),\left(b_{i}'\right)$ are mutually
indiscernible sequences in $\RV$ and $\rv\left(d-a_{\infty}\right)\models\widetilde{\phi}\left(\widetilde{x},\widetilde{b}_{0}\right)\land\widetilde{\phi}'\left(\widetilde{x},b_{0}'\right)$.
By $\inp$-minimality of $\RV$ we have that
either there is some $\alpha^{*}\models\left\{ \widetilde{\phi}'\left(\widetilde{x},b_{i}'\right)\right\} _{i\in\omega}$,
in which case we can find $d^{*}$ with $\rv\left(d^{*}-a_{\infty}\right)=\alpha^{*}$
and thus $d^{*}\models\left\{ \phi'\left(\rv\left(x-a_{i}'\right),b_{i}'\right)\right\} _{i\in\omega}$,
or that $\alpha^{*}\models\left\{ \widetilde{\phi}\left(\widetilde{x},\widetilde{b}_{i}\right)\right\} _{i\in\omega}$.
Then it follows from the definition of $\widetilde{\phi}$ that there
is $d^{*}$ satisfying $\rv\left(d^{*}-a_{\infty}\right)=\alpha^{*}$
and such that that $d^{*}\models\left\{ \phi\left(\rv\left(x-a_{i}\right),b_{i}\right)\right\} _{i\in\omega}$
--- a contradiction.

\smallskip \noindent
\textbf{Subcase 2}: Not Subcase 1.

\smallskip \noindent
Then we have the following observations.
\begin{claim}\label{cla: below}
For any $i,j \in \mathbb{Z}$ we have $\val(a_\infty - a_i) > \val(a'_j - a_i)$.
\end{claim}
\begin{proof}
Since $a'_j$ is not a pseudo-limit of the sequence $(a_i)$ (as we are not in Subcase 1), we must have $\val(a'_j - a_{i_1}) < \val(a_{i_2} - a_{i_1})$ for some $i_2 > i_1 \in \mathbb{Z}$. Then the claim follows by mutual indiscernibility.
\end{proof}

\begin{claim}
The sequence $(a'_i)$ must be pseudo-convergent.
\end{claim}
\begin{proof}
If $(a'_i)$ was a fan, in view of Claim \ref{cla: below} we would have $\val(a_i-a'_j)$ constant --- a contradiction since we are not in Case A. Hence it is pseudo-convergent, after possibly reversing the order, by Lemma \ref{lem: pseudo-conv or fan}.
\end{proof}

These two claims imply that the only
possibility is that $\left(a_{i}'\right)$ is pseudo-convergent and
that any $a_{i}$ is a pseudo-limit of it. But then reversing the
roles of the two sequences we are back to Subcase 1, concluding the analysis of the special case.

\medskip

Now we reduce the case of a general $\inp$-pattern to the special
case treated above. Assume that there is an $\inp$-pattern of depth
$2$. By Ramsey and compactness we may assume that the rows are mutually
indiscernible in the $L_{\ac}$-language. Though in Fact \ref{fac: Flenner's cell decomposition}
the formula defining $D$ may depend on the formula defining $S$,
by indiscernibility, Ramsey and compactness we may assume that the formulas in our $\inp$-pattern
 are in fact of the form $\phi\left(\rv\left(x-y_{1}\right),\ldots,\rv\left(x-y_{n}\right),z\right)$
and $\phi'\left(\rv\left(x-y_{1}\right),\ldots,\rv\left(x-y_{n}\right),z'\right)$,
for some $n\in\omega$, where $\phi$ and $\phi'$ are $\RV$-formulas. Let $d$ realize the first column of the $\inp$-pattern.

\textbf{Case 1}: $\val\left(d-a_{0,0}\right)<\val\left(a_{0,n}-a_{0,0}\right)$.
Then $\rv\left(d-a_{0,0}\right)=\rv\left(d-a_{0,n}\right)$ and we
define $\widetilde{\phi}\left(x,a_{i}\widetilde{b}_{i}\right)=\phi\left(\rv\left(x-a_{i,0}\right),\ldots,\rv\left(x-a_{i,n-1}\right),\rv\left(x-a_{i,0}\right), b_i\right)\land\val\left(x-a_{i,0}\right)<\val\left(a_{i,n}-a_{i,0}\right)$
with $\widetilde{b}_{i}=b_{i}\hat{\,}\rv\left(a_{i,n}-a_{i,0}\right)$.

\textbf{Case 2}: $\val\left(d-a_{0,0}\right)>\val\left(a_{0,n}-a_{0,0}\right)$.
Then $\rv\left(d-a_{0,n}\right)=\rv\left(a_{0,n}-a_{0,0}\right)$
and we define $\widetilde{\phi}\left(x,a_{i}\widetilde{b}_{i}\right)=\phi\left(\rv\left(x-a_{i,0}\right),\ldots,\rv\left(x-a_{i,n-1}\right),\rv\left(a_{i,n}-a_{i,0}\right), b_i \right)\land\val\left(x-a_{i,0}\right)>\val\left(a_{i,n}-a_{i,0}\right)$
with $\widetilde{b}_{i}=b_{i}\hat{\,}\rv\left(a_{i,n}-a_{i,0}\right)$.

\textbf{Case 3}: $v\left(d-a_{0,n}\right)<v\left(a_{0,n}-a_{0,0}\right)$
and \textbf{Case 4}: $v\left(d-a_{0,n}\right)>v\left(a_{0,n}-a_{0,0}\right)$
are symmetric to \textbf{Case 1} and \textbf{Case 2} respectively.

\textbf{Case 5}: $v\left(d-a_{0,0}\right)=v\left(d-a_{0,n}\right)=v\left(a_{0,n}-a_{0,0}\right)$.
Then $\rv\left(d-a_{0,0}\right)=\rv\left(d-a_{0,n}\right)+\rv\left(a_{0,n}-a_{0,0}\right)$.
We define 
\begin{eqnarray*}
\widetilde{\phi}\left(x,a_{i}\widetilde{b}_{i}\right) & = & \phi\left(\rv\left(x-a_{i,n}\right)+\rv\left(a_{i,n}-a_{i,0}\right),\ldots,\rv\left(x-a_{i,n-1}\right),\rv\left(x-a_{i,n}\right), b_i \right)\\
 &  & \land v\left(x-a_{i,n}\right)=v\left(a_{i,n}-a_{i,0}\right)\land \WD \left(\rv\left(x-a_{i,n}\right),\rv\left(a_{i,n}-a_{i,0}\right) \right)
\end{eqnarray*}
 with $\widetilde{b}_{i}=b_{i}\hat{\,}\rv\left(a_{i,n}-a_{i,0}\right)$.

In any of the cases, we still have that $\left(\widetilde{b}_{i}\right)_{i\in\mathbb{Z}},\left(b_{i}'\right)_{i\in\mathbb{Z}}$
are mutually indiscernible, that $d\models\widetilde{\phi}\left(x,a_{0}\widetilde{b}_{0}\right)\land\phi'\left(x,a_{0}'b_{0}'\right)$
and that $\left\{ \widetilde{\phi}\left(x,a_{i}\widetilde{b}_{i}\right)\right\} _{i\in\mathbb{Z}}$
is inconsistent. Thus we get a new $\inp$-pattern replacing $\left\{ \phi\left(x,a_{i}b_{i}\right)\right\} $
by $\left\{ \widetilde{\phi}\left(x,a_{i}\widetilde{b}_{i}\right)\right\} $,
with $\widetilde{\phi}$ involving one less term of the form $\rv\left(x-y_{i}\right)$.
Repeating the same operation $n$ times for $\phi$, and then for
$\phi'$, we reduce the situation to the special case of formulas
considered before.

\section{Relative quantifier elimination for $\RV$}

Now it will be more convenient to consider a valued field $K$ in
a slightly weaker language $L_{\RV}$. Namely, we associate with it
a three-sorted structure $\bar{K}=\left(K,\RV,\Gamma,\val_{\rv}\right)$
such that on $\RV$ we have the multiplicative group structure $\cdot,1$,
a constant $0$, a predicate for the residue field $k\subseteq\RV$ along with addition $\tilde{+}$ on $k$,
and a map $\val_{\rv}:\RV\to\Gamma$.

The partial addition relation $\oplus$ on $\RV$ is definable in $L_{\RV}$ (using \cite[Proposition 2.7]{flenner2011relative}):
$$\oplus(x,y,z) \iff \left( \val_{\rv}(x) < \val_{\rv}(y) \land z = x \right) \lor \left( \val_{\rv}(y) < \val_{\rv}(x) \land z = y \right) \lor$$
$$ \lor \left( \val_{\rv}(x) = \val_{\rv}(y) \land \left( \left(\frac{x}{y}\tilde{+}1 = 0 \land \val_{\rv}(z) > \val_{\rv}(x) \right) \lor \left( (\frac{x}{y} \tilde{+} 1) y = z  \land z \neq 0 \right) \right) \right).$$
The conclusion is that in particular if $\left(\RV,\Gamma,\val_{\rv}\right)$
is $\inp$-minimal as an $L_{\RV}$-structure, then $\left(\RV,\cdot,\oplus \right)$
is $\inp$-minimal as an $L_{\RV^{+}}$-structure. In the next section
we are going to demonstrate the former under the assumptions of the
main theorem, but in order to do that we prove a relative quantifier
elimination result for (a certain expansion of) the $L_{\RV}$ language.

\medskip\noindent
\textbf{Assumptions}
\begin{itemize}
\item $G$ is an abelian group such that $G/nG$ is finite for all $n<\omega$.
\item $K\subseteq G$ is a subgroup, with quotient $H=G/K$. Let $\pi:\, G\to H$
denote the projection map.
\item $M$ is the two-sorted structure with sorts $G$ and $H$, and the
following language.

\begin{itemize}
\item On $G$: we have the group structure $+,-,0$, a predicate $K\left(x\right)$
for the subgroup $K$, predicates $\left(P_{n}\left(x\right):n<\omega\right)$
interpreted as $P_{n}\left(x\right)\leftrightarrow\exists y\, ny=x$,
and constants naming a countable subgroup $G_{0}$ containing representatives
of each class of $G/nG$, for each $n<\omega$ (such that moreover
all classes of elements from $K$ are represented by elements from
$G_{0}\cap K$).
\item On $H$: we have some language $L_{H}$ (containing the induced group
structure) and we assume that the structure $\left(H,L_{H}\right)$
eliminates quantifiers.
\item On $K$: we have some language $L_{K}$ such that $\left(K,L_{K}\right)$
eliminates quantifiers and contains the language induced from $G$
(via the group structure and predicates $P_{n}$).
\item We have the projection group homomorphism $\pi:G\to H$.
\end{itemize}
\item Moreover, we assume that the language contains no other function symbols
apart from $\pi$ and the group structures on $G$ and $H$.
\item Finally, $H$ is torsion-free.\end{itemize}
\begin{prop}
\label{prop: Elimination of quantifiers in RV}$M$ has quantifier
elimination.\end{prop}
\begin{proof}
We prove it by back-and-forth. So assume that $M$ is $\aleph_1$-saturated
and we have two substructures $A$ and $B$ from $M$ and a partial
isomorphism $f:A\to B$. So $A,B\supseteq G_{0}$ contain elements
from both $G$ and $H$, both are closed under the group operations,
inverse and $\pi$.

Let $\alpha\in M$ be arbitrary, and we want to extend $f$ to be
defined on $A_{1}=A\left(\alpha\right)$, the substructure generated
by $\alpha A$. We assume that $\alpha\notin A$.

\medskip \noindent
\textbf{Step }1: If $\alpha\in H$, then we can extend $f$.

\smallskip
As $f|_{A\cap H}$ is $L_{H}$-elementary by quantifier elimination
in $\left(H,L_{H}\right)$, there is $\beta\in H$ and a partial $L_{H}$-automorphism
$g$ extending $f|_{A\cap H}$ and sending $ $$A\left(\alpha\right)\cap H$
to $B\left(\beta\right)\cap H$. Then we extend $f$ to $F$ defined
on $A\left(\alpha\right)$ by taking $F=f\cup g$ (note that, as there
are no functions from $H$ to $G$ in the language, $A\left(\alpha\right)\cap G=A\cap G$).

\smallskip
So by iterating Step 1 we may assume that $\alpha\in G$ and that
$\pi\left(a+n\alpha\right)\in A$ for all $a\in A$ and $n\in\mathbb{Z}$.

\smallskip \noindent
\textbf{Step 2}: Assume that $\alpha\in K$. Then we can extend $f$.

\smallskip
As $f|_{A\cap K}$ is $L_{K}$-elementary by quantifier elimination,
we can find $\beta\in K$ and a partial $L_{K}$-automorphism $g$
extending it and sending $A\left(\alpha\right)\cap K$ to $B\left(\beta\right)\cap K$.
Then we define $F$ on $A\left(\alpha\right)$ by setting $F\left(a+n\alpha\right)=f\left(a\right)+g\left(n\alpha\right)=f\left(a\right)+ng\left(\alpha\right)$
for all $a\in A$, $n\in\mathbb Z$ (note that $n\alpha\in A\left(\alpha\right)\cap K$
for all $n\in\mathbb Z$ by the assumption) and $F$ acts like $f$ on $A(\alpha)\cap H = A\cap H$.
\begin{itemize}
\item $F$ is well-defined: Assume that $a+n\alpha=a'+n'\alpha$, so $A\ni a-a'=\left(n'-n\right)\alpha$,
and thus $f\left(a\right)-f\left(a'\right)=f\left(a-a'\right)=f\left(\left(n'-n\right)\alpha\right)=\ldots$
as $\left(n'-n\right)\alpha\in K\cap A$ and $g|_{A\cap K}=f|_{A\cap K}$
$\ldots=g\left(\left(n'-n\right)\alpha\right)=ng\left(\alpha\right)-n'g\left(\alpha\right)$.
Then we have $F\left(a+n\alpha\right)-F\left(a'+n'\alpha\right)=f\left(a\right)+g\left(n\alpha\right)-f\left(a'\right)-g\left(n'\alpha\right)=0$.
\item $F$ extends $f$: immediate from the definition.
\item Note that $F|_{A\left(\alpha\right)\cap K}=g$, as given $a+n\alpha\in A\left(\alpha\right)\cap K$
it follows that $a\in A\cap K$, and as $f|_{A\cap K}=g|_{A\cap K}$
we have $F\left(a+n\alpha\right)=f\left(a\right)+g\left(n\alpha\right)=g\left(a\right)+g\left(n\alpha\right)=g\left(a+n\alpha\right)$.
\item $F|_G$ is a group homomorphism: $F\left(a+n\alpha+a'+n'\alpha\right)=F\left(\left(a+a'\right)+\left(n+n'\right)\alpha\right)=f\left(a+a'\right)+g\left(\left(n+n'\right)\alpha\right)=f\left(a\right)+f\left(a'\right)+g\left(n\alpha\right)+g\left(n'\alpha\right)=F\left(a+n\alpha\right)+F\left(a'+n'\alpha\right)$.
\item $F$ is onto $B(\beta)$: every element of $B\left(\beta\right)$ is of the form
$b+n\beta$, so $F\left(f^{-1}\left(b\right)+n\alpha\right)=b+n\beta$.
\item $F$ preserves $\pi$: On one hand $\pi\left(F\left(a+n\alpha\right)\right)=\pi\left(f\left(a\right)+ng\left(\alpha\right)\right)=\pi\left(f\left(a\right)\right)+n\pi\left(g\left(\alpha\right)\right)=\ldots$
as $g\left(\alpha\right)\in K$ $\ldots=\pi\left(f\left(a\right)\right)+0=f\left(\pi\left(a\right)\right)=F\left(\pi\left(a\right)\right)$
(recall that $\pi\left(a\right)\in A$). On the other hand $ $ we
have $F\left(\pi\left(a+n\alpha\right)\right)=F\left(\pi\left(a\right)+n\pi\left(\alpha\right)\right)=F\left(\pi\left(a\right)+0\right)=F\left(\pi\left(a\right)\right)$.
\item In particular, $F$ preserves $K\left(x\right)=\left\{ x\in G:\,\pi\left(x\right)=0\right\} $.
\item $F$ preserves $P_{k}$: $P_{k}\left(F\left(a+n\alpha\right)\right)\Leftrightarrow P_{k}\left(f\left(a\right)+ng\left(\alpha\right)\right)\Leftrightarrow P_{k}\left(a+ng\left(\alpha\right)\right)$
(as $f\left(a\right)=a\mod kG$) $\Leftrightarrow$ $P_{k}\left(a+n\alpha\right)$
(as $g\left(\alpha\right)=\alpha\mod kG$ because all representatives
of classes of $\alpha\in K$ are in $G_{0}\cap K\subseteq A\cap K$,
$P_{k}\cap K$ is $L_{K}$-definable and $g|_{A\left(\alpha\right)\cap K}$
is $L_{K}$-elementary).
\item $F$ preserves every $\phi(x_{1},\ldots,x_{k})\in L_{K}$: As $F|_{A\left(\alpha\right)\cap K}=g$
and $g$ is an $L_{K}$-elementary map.
\item $F$ preserves every $\psi\in L_{H}$: As $\pi\left(a+n\alpha\right)=\pi\left(a\right)+n\pi\left(\alpha\right)\in A\cap H$
(as $\pi\left(\alpha\right)\in A$ by the assumption), and $F|_{A\cap H}=f|_{A\cap H}$
is $L_{H}$-elementary.
\end{itemize}
So $F$ is a partial isomorphism as wanted.

\smallskip
By iterating Step 2 we may assume that $a+n\alpha\in K\Rightarrow a+n\alpha\in A$
for all $a\in A$ and $n\in\omega$.

\smallskip \noindent
\textbf{Step} \textbf{3}: Assume that $m\alpha\in A$ for some $m\geq1$.
Then we can extend $f$.

\smallskip
Let $m$ be minimal with this property.

\begin{claim} There is $\beta\in G$ satisfying $m\beta=f\left(m\alpha\right)$
and $\beta=\alpha\mod kG$ for all $k\in\omega$. \end{claim}
\begin{proof} By $\omega$-saturation it suffices to shows this one $k$ at a time.
By assumption there is some $g\in G_{0}$ such that $P_{k}\left(\alpha-g\right)$, then $P_k\left(\alpha-g\right)\Rightarrow P_{mk}\left(m\alpha-mg\right)\Rightarrow P_{mk}\left(f\left(m\alpha\right)-mg\right)$
(as $m\alpha,mg\in A$, $f\left(mg\right)=mf\left(g\right)=mg$ and
$f$ preserves $P_{l}$ for all $l<\omega$) $\Rightarrow$ $\exists\gamma\in G$
such that $mk\gamma=f\left(m\alpha\right)-mg$. Let $\beta=k\gamma+g$.
Then $m\beta=f\left(m\alpha\right)$ and $\beta=g=\alpha\mod kG$,
and the claim is proved. \end{proof}

We define $F$ on $A\left(\alpha\right)\cap G$ by setting $F\left(a+n\alpha\right)=f\left(a\right)+n\beta$
and $F|_{A\left(\alpha\right)\cap H}=f|_{A\left(\alpha\right)\cap H}$
as $A\left(\alpha\right)\cap H=A\cap H$.
\begin{itemize}
\item $F$ is well-defined: If $a+n\alpha=a'+n'\alpha$ with $a,a'\in A$,
then $\left(n-n'\right)\alpha=a'-a\in A$. It follows that $m$ divides
$\left(n-n'\right)$ by minimality (assume that $n-n'=km+m_{1}$, $|m_1|<m$, then $m_{1}\alpha=a'-a-km\alpha\in A$, contradiction), say $\left(n-n'\right)=km$. Thus $f\left(a'\right)-f\left(a\right)=f\left(a'-a\right)=f\left(\left(n-n'\right)\alpha\right)=f\left(km\alpha\right)=kf\left(m\alpha\right)=km\beta=\left(n-n'\right)\beta$.
But then $F\left(a+n\alpha\right)-F\left(a'+n'\alpha\right)=f\left(a\right)+n\beta-f\left(a'\right)-n'\beta=0$.
\item $F$ extends $f$ is obvious from the definition.
\item $F$ is a group homomorphism from $A\left(\alpha\right)$ to $B\left(\beta\right)$:

$F\left(\left(a+n\alpha\right)+\left(a'+n'\alpha\right)\right)=F\left(\left(a+a'\right)+\left(n+n'\right)\alpha\right)=f\left(a+a'\right)+\left(n+n'\right)\beta=\left(f\left(a\right)+n\beta\right)+\left(f\left(a'\right)+n'\beta\right)=F\left(a+n\alpha\right)+F\left(a'+n'\alpha\right)$.
\item $F$ preserves $\pi$: First observe that $\pi\left(m\beta\right)=\pi\left(f\left(m\alpha\right)\right)$,
so $m\pi\left(\beta\right)=\pi\left(f\left(m\alpha\right)\right)\overset{\mbox{as }m\alpha\in A}{=}f\left(\pi\left(m\alpha\right)\right)=f\left(m\pi\left(\alpha\right)\right)=mf\left(\pi\left(\alpha\right)\right)$,
and as $H$ is torsion free this implies that $\pi\left(\beta\right)=f\left(\pi\left(\alpha\right)\right)$.
But then $F\left(\pi\left(a+n\alpha\right)\right)=f\left(\pi\left(a+n\alpha\right)\right)=f\left(\pi\left(a\right)+n\pi\left(\alpha\right)\right)=f\left(\pi\left(a\right)\right)+nf\left(\pi\left(\alpha\right)\right)=\pi\left(f\left(a\right)\right)+n\pi\left(\beta\right)=\pi\left(f\left(a\right)+n\beta\right)=\pi\left(F\left(a+n\alpha\right)\right)$.
\item In particular, $F$ preserves $K\left(x\right)=\left\{ x\in G:\,\pi\left(x\right)=0\right\} $.
\item $F$ preserves $P_{k}$$\left(x\right)$: By the choice of $\beta$
we have $\alpha=\beta\mod kG$ for all $k$, and for any $a\in A$
we have $f\left(a\right)=a\mod kG$ for all $k$ (as $G_{0}\subseteq A$
and $f$ preserves $P_{k}$), hence $P_{k}\left(F\left(a+n\alpha\right)\right)\Leftrightarrow P_{k}\left(f\left(a\right)+n\beta\right)\Leftrightarrow P_{k}\left(a+n\alpha\right)$.
\item $F$ preserves $L_{K}$-formulas: As $a+n\alpha\in K\Rightarrow a+n\alpha\in A$
by the assumption and $F|_{A\cap K}=f|_{A\cap K}$ is $L_{K}$-elementary
by elimination of quantifiers in $\left(K,L_{K}\right)$.
\item $F$ preserves $L_{H}$-formulas: As $F|_{A\left(\alpha\right)\cap H}=f|_{A\left(\alpha\right)\cap H=A\cap H}$
by definition, and $f$ is $L_{H}$-elementary.
\end{itemize}

\noindent
So we may assume that:
\begin{enumerate}
\item $A\cap H$ is a relatively divisible subgroup of $H$ (iterating Step 1);
\item $A\cap G$ is a relatively divisible subgroup of $A(\alpha)\cap G$ (iterating Step 3);
\item $\pi\left(a+n\alpha\right)\in A$ for all $a\in A,n\in\mathbb Z$ (iterating
Step 1);
\item $a+n\alpha\notin K$ for all $a\in A,n\in\mathbb{Z}\setminus\left\{ 0\right\} $
(as $a+n\alpha\in K\Rightarrow a+n\alpha\in A$ by Step 2, so $n\alpha\in A$,
so $\alpha\in A$ by divisibility of $A$ --- contradicting the assumption).
\end{enumerate}

\smallskip\noindent
\textbf{Step }4: General case.

\begin{claim}There is some $\beta\in G$ such that $\pi\left(\beta\right)=f\left(\pi\left(\alpha\right)\right)$
and $\alpha=\beta\mod kG$ for all $k\in\omega$.\end{claim}
\begin{proof}
By $\omega$-saturation we only need to consider one value of $k$
at a time. Let $g\in G_{0}$ be such that $P_{k}\left(g+\alpha\right)$
holds, then $\pi\left(g+\alpha\right)$ is $k$-divisible as well.
As $g\in A\Rightarrow g+\alpha\in A\left(\alpha\right)\Rightarrow\pi\left(g+\alpha\right)\in A\cap H$
and $f|_{A\cap H}$ is $L_{H}$-elementary, it follows that $f\left(\pi\left(g+\alpha\right)\right)$
is $k$-divisible as well. Take $\beta$ to be $k\beta'-g$ where
$\pi\left(\beta'\right)=\frac{f\left(\pi\left(g+\alpha\right)\right)}{k}$ (recall that $H$ is torsion free).
Now we have $P_{k}\left(g+\beta\right)$ and $\pi\left(\beta\right)=k\pi\left(\beta'\right)-\pi\left(g\right)=f\left(\pi\left(g+\alpha\right)\right)-\pi\left(g\right)=f\left(\pi\left(g\right)\right)+f\left(\pi\left(\alpha\right)\right)-\pi\left(g\right)=f\left(\pi\left(\alpha\right)\right)$
as $f\left(\pi\left(g\right)\right)=\pi\left(f\left(g\right)\right)$
and $f\left(g\right)=g$, so the claim is proved.\end{proof}

We define $F\left(a+n\alpha\right)=f\left(a\right)+n\beta$ and $F|_{A\left(\alpha\right)\cap H=A\cap H}=f|_{A\cap H}$.
\begin{itemize}
\item $F$ is well-defined: If $a+n\alpha=a'+n'\alpha$, then $\left(a-a'\right)+\left(n-n'\right)\alpha=0\in A$,
which implies by the assumption that $n=n'$ and $a=a'$.
\item $F$ is a homomorphism: clear from definition and as $f$ is a homomorphism
on $A$.
\item $F$ preserves $\pi$ (so in particular $K$): $\pi\left(F\left(a+n\alpha\right)\right)=\pi\left(f\left(a\right)+n\beta\right)=\pi\left(f\left(a\right)\right)+n\pi\left(\beta\right)=f\left(\pi\left(a\right)\right)+nf\left(\pi\left(\alpha\right)\right)=f\left(\pi\left(a\right)+n\pi\left(\alpha\right)\right)=f\left(\pi\left(a+n\alpha\right)\right)=F\left(\pi\left(a+n\alpha\right)\right)$.
\item $F$ preserves $P_{k}$: $P_{k}\left(F\left(a+n\alpha\right)\right)\Leftrightarrow P_{k}\left(f\left(a\right)+n\beta\right)\Leftrightarrow$
$P_{k}\left(a+n\beta\right)$ (as $f\left(a\right)=a\mod kG$ because
we have all the representatives in $G_{0}$) $\Leftrightarrow P_{k}\left(a+n\alpha\right)$
(as $\alpha=\beta\mod kG$ by the choice of $\beta$).
\item $F$ preserves $L_{K}$-formulas and $L_{H}$-formulas: as in Step
3.\qedhere
\end{itemize}
\end{proof}
\begin{cor}
$H$ and $K$ are fully stably embedded, i.e. any subset of $H$ (resp. $K$) definable with external parameters is already definable with internal parameters in $L_H$ (resp., $L_K$) --- this follows directly from the elimination of quantifiers.\end{cor}

\section{Reduction from $\RV$ to $k$ and $\Gamma$}

\begin{prop}
\label{prop: concluding inp-minimality}Let $M=\left(G,K,H\right)$
be a structure satisfying the assumptions from the previous section.
Assume moreover that:
\begin{enumerate}
\item $K$ (viewed as an $L_{K}$ structure) and $H$ (viewed as an $L_{H}$
structure) are both $\inp$-minimal;

\item for every $n$, there are only finitely many $x\in G$ for which $nx=0$ (since $H$ is torsion-free, such elements are in fact in $K$).
\end{enumerate}
Then $M$ is $\inp$-minimal.\end{prop}
\begin{proof}

We are working in a saturated extension of $M$. Assume that the conclusion fails, then
we have
an $\inp$-pattern $\phi\left(x,y\right),\phi'\left(x,y'\right),\bar{a}=\left(a_{i}\right),\bar{a}'=\left(a_{i}'\right)$
witnessing this, with $\bar{a}$ and $\bar{a}'$ mutually indiscernible. 
In particular they are mutually indiscernible over $G_{0}\subseteq\acl\left(\emptyset\right)$
which contains representatives of each class of $G/nG$ and all torsion
of $G$, and rows are $k_{*}$-inconsistent. Let $b\models\phi\left(x,a_{0}\right)\land\phi'\left(x,a_{0}'\right)$. 
 It follows from quantifier elimination that $\phi\left(x,a_{i}\right)$ is equivalent
to a disjunction of conjuncts of the form $\theta(t_{i,0}(x),\ldots,t_{i,l-1}(x), \alpha_i) \land \psi\left(\pi\left(x\right),b_{i}\right)\land\chi\left(x,c_{i}\right)\land\rho\left(x,e_{i}\right)$ where:

\begin{itemize}
	\item the $t_{i,j}$ are terms with parameters in $G$, $\alpha_i \in K$ and $\theta$ is an $L_K$-formula;
	\item $\psi$ is an $L_H$-formula and $b_i \in H$;
	\item $\chi\left(x,c_{i}\right)$
is of the form $\bigwedge_{j<k}n_{j}x+c_{i,j}=0\land\bigwedge_{j<k}m_{j}x+d_{i,j}\neq0$
with $c_{i}=\left(c_{i,j}\right)_{j<k}\hat{}\left(d_{i,j}\right)_{j<k}$ from $G$;

\item $\rho\left(x,e_{i}\right)$ is of the form $$\bigwedge_{j<k}P_{m_{j}'}\left(n_{j}'x+e_{i,j}'\right)\land\bigwedge_{j<k}\neg P_{m_{j}''}\left(n_{j}''x+e_{i,j}''\right)$$
with $e_{i}=\left(e_{i,j}'\right)_{j<k}\hat{\,}\left(e_{i,j}''\right)_{j<k}$.

\end{itemize}
Forgetting all but one disjunct satisfied by $b$, we may assume that $\phi(x,a_i)$ is equal to such a conjunction.

Any term $t_{i,j}$ is of the form $n_{i,j}x - g_{i,j}$ and the formula makes sense only when $n_{i,j}x - g_{i,j}\in K$, that is when $\pi(x) = \pi(g_{i,j})/n_{i,j}$. Choose some $h_i$ such that $\pi(h_i)=\pi(g_{i,j})/n_{i,j}$ for some/all $j$. We can then replace $n_{i,j}x - g_{i,j}$ with $n (x-h_{i}) + h'_{i,j}$ with $h'_{i,j}\in K$. Adding $h'_{i,j}$ to $\alpha_i$ and changing the formula $\theta$, we replace $\theta$ by a formula $\theta'(x-h_i,\alpha'_i)$, $\theta'\in L_K$.

Recalling that $G/nG$ is finite for every $n<\omega$, $\rho\left(x,e_{i}\right)$ is equivalent to some finite disjunction
of the form $\bigvee_{i<N}P_{k_{i}}\left(x-g_{i}\right)$ where $g_{i}\in G_{0}$
(so for example to express $\neg P_{k}\left(nx+e\right)$ we have
to say that $x$ belongs to one of the finitely many classes $\mod kG$
satisfying this, and to express $P_{k}\left(nx+e\right)\land P_{l}\left(n'x+e'\right)$
we have to say that $x$ belongs to a certain subset of the classes
$\mod klG$).

Note that $\chi\left(x,c_{0}\right)$ is infinite as $\chi\left(x,c_{0}\right)\land\phi'\left(x,a_{i}\right)$
is consistent for every $i\in\omega$, while $\left\{ \phi'\left(x,a_{i}\right)\right\} _{i\in\omega}$
is $k_{*}$-inconsistent. Thus $\chi\left(x,c_{i}\right)$ can only
be of the form $\bigwedge_{j<k}n_{j}x+c_{i,j}\neq0$ (as every equation
of the form $nx+c=0$ has only finitely many solutions by assumption (2)).

Thus we may assume that $\phi(x,a_i) = \theta(x - h_i, \alpha_i) \land \psi\left(\pi\left(x\right),b_{i}\right)\land\chi\left(x,c_{i}\right)\land P_{l}\left(x-g\right)$ where:

\begin{itemize}
	\item $\alpha_i \in K$ and $\theta$ is an $L_K$-formula, 
	\item $\psi$ is an $L_H$-formula and $b_i \in \Gamma$, 
	\item $\chi\left(x,c_{i}\right) = \left(\bigwedge_{j<k}n_{j}x+c_{i,j}\neq 0\right)$
    \item $l \in \omega, g \in G_{0}$.

\end{itemize}

Similarly, we may assume that $\phi'(x,a'_i) = \theta'(x - h'_i, \alpha'_i) \land \psi' \left(\pi\left(x\right),b'_{i}\right)\land\chi' \left(x,c'_{i}\right)\land P_{l'}\left(x-g' \right)$ with the same properties.

\smallskip \noindent
\textbf{Case 1}: $b\in H$. Then by full stable embeddedness of $H$ we
can replace our array by $\widetilde{\phi}\left(x,\widetilde{a}_{i}\right)$
and $\widetilde{\phi}'\left(x,\widetilde{a}_{i}'\right)$ where $\widetilde{\phi},\widetilde{\phi}'\in L_{H}$
and $\widetilde{a}_{i},\widetilde{a}_{i}'\in H$ are such that $\widetilde{\phi}\left(x,\widetilde{a}_{i}\right)\cap H\left(x\right)=\phi\left(x,a_{i}\right)\cap H\left(x\right)$,
and similarly for $\widetilde{\phi}'$. But this contradicts $\inp$-minimality
of $\left(H,L_{H}\right)$.

\smallskip \noindent
\textbf{Case 2}: $b\in K$. Similarly, by full stable embeddedness of $K$ we can replace our array by $\widetilde{\phi}\left(x,\widetilde{a}_{i}\right)$
and $\widetilde{\phi}'\left(x,\widetilde{a}_{i}'\right)$ where $\widetilde{\phi},\widetilde{\phi}'\in L_{K}$
and $\widetilde{a}_{i},\widetilde{a}_{i}'\in K$ are such that $\widetilde{\phi}\left(x,\widetilde{a}_{i}\right)\cap K\left(x\right)=\phi\left(x,a_{i}\right)\cap K\left(x\right)$,
and similarly for $\widetilde{\phi}'$. But this contradicts $\inp$-minimality
of $\left(K,L_{K}\right)$.

\smallskip \noindent
\textbf{Case 3}: $b\notin K \cup H$.

\noindent
\textbf{Subcase 3.1} Neither $\theta$ occurs in $\phi$ nor $\theta'$ occurs in $\phi'$ (i.e.~$\phi$ is equivalent to the formula obtained from it by omitting $\theta$).

Then we have $\phi(x,a_i) = \psi\left(\pi\left(x\right),b_{i}\right)\land\chi\left(x,c_{i}\right)\land P_{l}\left(x-g\right)$ and $\phi'(x,a'_i) = \psi' \left(\pi\left(x\right),b'_{i}\right)\land\chi' \left(x,c'_{i}\right)\land P_{l'}\left(x-g' \right)$.

Consider $\widetilde{\psi}(x',b_i) := \psi(x',b_i) \land$ ``$x'-\pi\left(g\right)$ is $l$-divisible'' and $\widetilde{\psi}'(x',b'_i) := \psi(x',b'_i) \land$ ``$x'-\pi\left(g'\right)$ is $l'$-divisible'' --- this is an array in the structure induced on $H$. Note that $\pi\left(b\right)\models \widetilde{\psi}\left(x',b_{0}\right) \land \widetilde{\psi}'\left(x',b_{0}'\right)$.

\noindent
\textbf{Subcase 3.1(a).} $K$ is infinite.

As $H$ is $\inp$-minimal, it follows without loss of generality that
the set $\left\{ \widetilde\psi\left(x',b_{i}\right) : i<\omega \right\}$
has a solution $h$ in $H$.

Say $h-\pi\left(g\right)=l\gamma$. Take $\beta\in G$ such that
$\pi\left(\beta\right)=\gamma$. As $K$ is infinite, there is an
infinite sequence $\left(\beta_{i}\right)_{i\in\omega}$ in $K$ such
that all the differences $\beta_{i}-\beta_{j}$ are pairwise different.
Let $e_{i}'=\beta+\beta_{i}$. Then we still have that $e_{i}'-e_{j}'$
are all pairwise different, and that $\pi\left(e_{i}'\right)=\pi\left(\beta\right)+\pi\left(\beta_{i}\right)=\gamma$.
Note that as by assumption there are only finitely many $l$-torsion
elements in $G$, we may assume that $e_{i}'-e_{j}'$ is not $l$-torsion,
for any $i\neq j$.

Finally, define $e_{i}=le_{i}'+g$. We have:
\begin{itemize}
\item all $e_{i}$'s are pairwise different (as $e_{i}=e_{j}\Rightarrow\left(e_{i}'-e_{j}'\right)$
is $l$-torsion, contradicting the choice of the elements $b_{i}'$).
\item $\pi\left(e_{i}\right)=l\pi\left(e_{i}'\right)+\pi\left(g\right)=l\gamma+\pi\left(g\right)=h$.
\item $P_{l}\left(e_{i}-g\right)$ holds as $e_{i}-g=le_{i}'$.
\end{itemize}
As the set $\bigvee_{i<k_{*}+1}\left(\bigvee_{j<k}n_{j}x+c_{i,j}=0\right)$
is finite, then one of the $e_{i}$'s realizes the first $k_{*}$ elements of
the first row --- a contradiction.

\noindent
\textbf{Subcase 3.1(b).} $K$ is finite.

It follows that all of the fibers of $\pi$ are finite.
\begin{claim}
One of the partial types $\{ \widetilde{\psi}(x',b_i) : i \in \omega \}$ or $\{  \widetilde{\psi}'(x',b'_i) : i \in \omega \}$ has infinitely many solutions in $H$.
\end{claim}
\begin{proof}
By $\inp$-minimality of $H$ we find some $e_{0}'\in H$
a solution to one of the rows $\{ \widetilde{\psi}(x',b_i) : i \in \omega \}$ or $\{  \widetilde{\psi}'(x',b'_i) : i \in \omega \}$. By
Ramsey, mutual indiscernibility and compactness we can find some $e_{0}\in H$
which is still a solution to the same row, and moreover $\bar{b},\bar{b}'$
are mutually indiscernible over $e_{0}$, so we can add it to the
base. Let $\widetilde{\psi}_{1}\left(x',b_{i}\right):=\widetilde{\psi}\left(x',b_{i}\right)\land x'\neq e_{0}$,
and the same for $\widetilde{\psi}_{1}'$.

As by assumption and mutual indiscernibility $ \psi\left(\pi\left(x\right),b_{0}\right) \land P_{l}\left(x-g\right) \land \phi'(x,a_i)$ is consistent for each $i \in \omega$, and $\left\{ \phi'\left(x,a_{i}\right)\right\} _{i\in\omega}$
is $k_{*}$-inconsistent, it follows that for infinitely many $i \in \omega$, we can find pairwise different $f_i \models \psi\left(\pi\left(x\right),b_{0}\right) \land P_{l}\left(x-g\right) \land \phi'(x,a_i)$. As all fibers of $\pi$ are finite, this implies that in fact in $H$ for infinitely many $i$'s we can find pairwise different $f'_i \models \widetilde{\psi}(x',b_0) \land \widetilde{\psi}'(x',b'_i)$. Thus $\widetilde{\psi}_{1}\left(x',b_{0}\right) \land \widetilde{\psi}'_{1}\left(x',b'_{i}\right)$ is consistent for some $i$, and so $\widetilde{\psi}_{1}\left(x',b_{0}\right) \land \widetilde{\psi}'_{1}\left(x',b'_{0}\right)$ is consistent by mutual indiscernibility over $e_0$. 
Repeating this argument, by induction on $s \in \omega$ we can choose $e_s \in H$ such that each $e_{s+1}$ satisfies one of the rows of the array $\{ \widetilde \psi_{s+1}(x',b_i) : i \in \omega\}, \{ \widetilde \psi'_{s+1}(x',b'_i) : i \in \omega\}$, with $\widetilde \psi_{s+1}(x',b_i) := \widetilde \psi_{s}(x',b_i) \land x' \neq e_s$ and $\widetilde \psi'_{s+1}(x',b'_i) := \widetilde \psi'_{s}(x',b'_i) \land x' \neq e_s$.
In particular, all $e_s$ are pairwise distinct, and by pigeonhole infinitely many of them realize the same row, so in particular the same row of the original array.
\end{proof}

So let now $(e_i : i \in \omega)$ be an infinite list of pairwise different solutions of $\{ \widetilde{\psi}(x',b_i) : i \in \omega \}$ in $H$. In particular $e_i - \pi(g) = l \gamma_i$ for some $\gamma_i \in H$ with $(\gamma_i : i \in \omega)$ pairwise different. Let $\beta_i \in G$ be arbitrary such that $\pi(\beta_i) = \gamma_i$. As all fibers of $\pi$ are finite, we may assume that all of $\beta_i$'s are pairwise different as well. Finally, let $f_i := l \beta_i + g$. We have:
\begin{itemize}
\item $(f_i : i \in \omega)$ are pairwise different,
\item $P_l(f_i - g)$ holds for all $i \in \omega$, as $f_i - g = l \beta_i$,
\item $\pi(f_i) = l \pi(\beta_i) + \pi(g) = l \gamma_i + \pi(g) = e_i$.
\end{itemize}

As the set $\bigvee_{i<k_{*}+1}\left(\bigvee_{j<k}n_{j}x+c_{i,j}=0\right)$
is finite, then one of the $f_{i}$'s realizes at least $k_{*}$ elements of
the first row --- a contradiction.

\smallskip \noindent
\textbf{Subcase 3.2} $\theta$ occurs in $\phi$ and $\theta'$ occurs in $\phi'$. I.e.~$\phi$ (respectively, $\phi'$) is not equivalent to the formula obtained from it by omitting $\theta$ (respectively, $\theta'$).

Syntactically, this is only possible if $b - g_0 \in K, b - g'_0 \in K$, hence both $\pi(b) \in \dcl(g_0)$ and $\pi(b) \in \dcl(g'_0)$. By mutual indiscernibility of the rows it follows that $\bar{a}, \bar{a}'$ are mutually indiscernible over $\pi(b)$ and we can add it to the base.

Then by mutual indiscernibility of $\bar{a}, \bar{a}'$ over $\pi(b)$, Ramsey, compactness and applying an automorphism, we can find some $f\in G$ such that $\pi(f) = \pi(b)$ and $\bar{a}, \bar{a}'$ are mutually indiscernible over $f$. So we can add $f$ to the base as well.

Taking $c := b-f$ we have $c \in K$. Translating by $f$, we can consider a new array $\widetilde{\phi}\left(x,\widetilde{a}_{i}\right), \widetilde{\phi}'\left(x,\widetilde{a}_{i}'\right)$ where $\widetilde{\phi}(x, a_i) = \theta(x+f - h_i, \alpha_i) \land \psi\left(\pi\left(x + f \right),b_{i}\right)\land\chi\left(x + f,c_{i}\right)\land P_{l}\left(x + f -g\right)$, and analogously for $\widetilde{\phi}'$. Note that the first column is realized by $c \in K$. By Case 2, we can find some $c'$ realizing, say, the first row of the new array. But then taking $b' := c' + f$ clearly $b'$ realizes the first row of the old array.

\smallskip \noindent
\textbf{Subcase 3.3} $\theta$ occurs in $\phi$, but $\theta'$ does not occur in $\phi'$ (and the symmetric case by permuting the rows).

By assumption $\phi'(x,a'_i) = \psi' \left(\pi\left(x\right),b'_{i}\right)\land\chi' \left(x,c'_{i}\right)\land P_{l'}\left(x-g' \right)$.
As in Subcase 3.1, it follows that $\pi(b) \in \dcl(a_0)$, say $\pi(b) = f(a_0)$ for some $\emptyset$-definable function $f$. We have $b \models \phi'(x,a'_0)$. In particular, $\models \psi'(f(a_0), b'_0) \land \textrm{``}f(a_0) - \pi(g') \textrm{ is } l' \textrm{-divisible''}$.
By mutual indiscernibility of $\bar{a}, \bar{a}'$ it follows that $\models \psi'(f(a_i), b'_j) \land \textrm{``}f(a_i) - \pi(g') \textrm{ is } l' \textrm{-divisible''}$ for all $i,j \in \omega$.

We may also assume that all of $\{ f(a_i) : i \in \omega  \}$ are pairwise different. Otherwise, if $f(a_i) = f(a_j)$ for some $i<j$, by indiscernibility $\pi(b) = f(a_0) = f(a_\infty)$, and so $\bar{a}, \bar{a}'$ are mutually indiscernible over $\pi(b)$ --- and we can conclude as in Subcase 3.2.
It follows that the partial type $\{ \psi'(x', b'_j) \land \textrm{``} x' - \pi(g') \textrm{ is } l' \textrm{-divisible''}  \}$ has infinitely many solutions in $H$, witnessed by $\{ f(a_i) : i\in \omega \}$.
Now this implies that the second row of the original array $\{ \phi'(x,a'_i) : i \in \omega  \}$ is consistent. Namely, if $K$ is infinite, then we conclude as in Case 3.1(a) using one of the solutions, and if $K$ is finite we conclude as in Case 3.1(b).

\end{proof}

\textbf{Proof of Theorem \ref{thm: main}.} Given a valued field $\bar{K}$
satisfying the assumption of Theorem \ref{thm: main}, via the reductions
in Sections 2 and 3 it is enough to demonstrate that $\left(\RV,k,\Gamma\right)$
is $\inp$-minimal. For this it is enough to show that the assumptions
of Proposition \ref{prop: concluding inp-minimality} are satisfied
for $G=\RV$, $K$ a Morleyzation of $k$ and $H$ a Morleyzation
of $\Gamma$. Both $K$ and $H$ are $\inp$-minimal as Morleyzation
obviously preserves $\inp$-minimality, $H$ is torsion-free since
$\Gamma$ is an ordered abelian group.

As $\Gamma$ is an $\inp$-minimal ordered group, it follows from
\cite[Lemma 3.2]{simon2011dp} that $\Gamma/n\Gamma$ is finite for
all $n\in\omega$. Besides, we have that $k^\times / (k^\times)^p$ is finite for all prime $p$ by assumption.
Therefore also $\RV/n\RV$ is finite for all $n$. Finally, $k^{\times}$ has finite $n$-torsion for all $n$.


\subsection*{Remarks and questions}
We do not know if the assumption that $k^\times/ (k^\times)^p$ is finite for all $p$ is in fact necessary. It follows from the proof of \cite[Corollary 4.6]{chernikov2015groups} that if $k$ is an $\inp$-minimal field, then there can be at most one prime $p$ for which $k^\times/ (k^\times)^p$ is infinite.

\begin{problem}
Let $k$ be an $\inp$-minimal field. Is it true that $k^\times / (k^\times)^p$ is finite for all prime $p$? Or at least, can we omit this extra assumption from Theorem \ref{thm: main}?
\end{problem}

The answer is positive for a dp-minimal field by the results of Johnson \cite{johnson2018canonical} (so under the assumptions of Theorem \ref{thm: main}, we have that $\bar{K}$ is dp-minimal if and only if both $k$ and $\Gamma$ are dp-minimal), but the proof relies on the construction of a valuation which doesn't seem to be available in the general $\inp$-minimal case.

Another natural direction is to generalize Theorem \ref{thm: main} from the case of burden $1$ to a general burden calculation.
\begin{problem}
Let $\bar{K}=\left(K,\RV,\rv\right)$ be a Henselian
valued field of equicharacteristic $0$, viewed as a structure in
the $\RV$-language. Is it true that $\bdn(\bar{K}) = \max \{ \bdn(k), \bdn(\Gamma) \}$?\footnote{While the article was under review, this question was answered positively by Touchard \cite{touchard2018burden}.}

\end{problem}

\subsection*{Acknowledgements}
We thank Martin Hils and the anonymous referee for their comments on an earlier version of the paper.

The research leading to this paper has been partially supported by the European Research Council under the European Union's Seventh Framework Programme (FP7/2007-2013)/ERC Grant Agreement No. 291111 and by ValCoMo (ANR-13-BS01-0006).

Chernikov was partially supported by the Fondation Sciences Mathematiques de Paris (ANR-10-LABX-0098), by the NSF (grants DMS-1600796 and DMS-1651321), and by the Sloan Foundation.

Simon was partially supported by NSF (grant DMS-1665491) and the Sloan Foundation.

\bibliography{ultraproductQp}

\end{document}